\documentclass[a4paper,11pt]{article}
\usepackage[sc]{mathpazo}
\usepackage[T1]{fontenc}
\usepackage{amsthm,amsmath, amsfonts,hyperref,a4wide}
\usepackage[usenames,dvipsnames]{xcolor} 
\usepackage{tikz}
\makeatletter
\tikzset{my loop/.style =  {to path={
  \pgfextra{}
  [looseness=12,min distance=6mm]
  \tikz@to@curve@path},font=\sffamily\small
  }}  
\makeatletter 
\usetikzlibrary{matrix,arrows,backgrounds,positioning,fit,shapes}

\newtheorem{theorem}{Theorem}
\newtheorem{lemma}[theorem]{Lemma}
\newtheorem{cor}[theorem]{Corollary}

\theoremstyle{definition}

\theoremstyle{remark}
\newtheorem{rem}{Remark}

\newtheorem*{ack}{Acknowledgements}

\newcommand*{\R}{\mathbb{R}}

\newcommand*{\N}{\mathbb{N}}
\newcommand*{\C}{\mathbb{C}}

\newcommand*{\F}{\mathcal{F}}
\newcommand*{\G}{\mathcal{G}}

\newcommand*{\rk}{\mathrm{rk}}
\newcommand*{\tr}{\mathrm{tr}}

\newcommand*{\stab}{\mathrm{Stab}}

\newcommand*{\ev}{\text{ev}}

\newcommand{\la}{\langle}
\newcommand{\ra}{\rangle}
\newcommand{\field}{\mathbb{F}}

\newcommand{\pr}{\text{pr}}

\title{Edge-reflection positivity and weighted graph homomorphisms\footnote{An extended abstract of this work appeared in the Proceedings of the 7th European Conference on Combinatorics, Graph Theory and Applications (EuroComb `13) entitled: 'A characterization of edge-reflection positive partition functions of vertex-coloring models'. }}
\author{Guus Regts\footnote{University of Amsterdam. Email: \texttt{guusregts@gmail.com.} This work was started when the author was at CWI, Amsterdam; it is based on Chapter 7 of his PhD thesis \cite{R13}. } }

\begin{document}
\maketitle
\abstract{B. Szegedy [Edge coloring models and reflection positivity, {\sl Journal of
the American Mathematical Society} {\bf 20} (2007) 969--988] showed that the number of homomorphisms into a weighted graph is equal to the partition 
function of a complex edge-coloring model.
Using some results in geometric invariant theory, we characterize for which weighted graphs the edge-coloring 
model can be taken to be real valued that is, we characterize for which weighted graphs the number
of homomorphisms into them are edge-reflection positive. 
In particular, we determine explicitly for which simple graphs the number of homomorphisms into them is equal to the partition function of a real edge-coloring model.
This answers a question posed by Szegedy.

\section{Introduction}\label{sec:intro}
Partition functions of vertex and edge-coloring models are graph invariants 
introduced by de la Harpe and Jones \cite{HJ}. In fact, in \cite{HJ} they are called spin and vertex models respectively. 
(Partition functions of vertex-coloring models are exactly the number of homomorphisms into weighted graphs, as we will see in Section \ref{sec:prelim}.)
Both models give a rich class of graph invariants.
But they do not coincide. For example the number of matchings in a graph is the partition function of a real 
edge-coloring model but not the partition function of any real vertex-coloring model. 
This can be deduced from the characterization of partition functions of real 
vertex-coloring models by Freedman, Lov\'asz and Schrijver \cite{FLS}. (It is neither the partition 
function of any complex vertex-coloring model, but we will not prove this here.)
Conversely, the number of independent sets is not the partition function of any real edge-coloring model, as 
follows from Szegedy's characterization of partition functions of real edge-coloring models \cite{SZ7}, but it is equal to the number of homomorphisms into 
\tikzstyle{vertex}=[circle,draw=black!100,fill=black, scale=.4]
\begin{tikzpicture}[every loop/.style={}]
\node[vertex](a){};
\node[vertex, right of =a, node distance=1.2cm](b){};
\node[left of =a, node distance=.5cm](c){};
\path (a) edge[in=120, out=180,loop] (a);
\draw (a) to (b);
\end{tikzpicture}
\quad (i.e., it is the partition function of a (real) vertex-coloring model).

However, Szegedy \cite{SZ7} showed that the partition function of any vertex-coloring model can be 
obtained as the partition function of a complex edge-coloring model.
Moreover, he gave examples when the edge-coloring model can be taken to be real valued. 
This made him ask the question which partition functions of real vertex-coloring models are partition 
functions of real edge-coloring models (cf. \cite[Question 3.2]{SZ7}).
In fact, he phrased his question in terms of edge-reflection positivity.
We will get back to that in Section \ref{sec:edge-ref}.

In this paper we completely characterize for which vertex-coloring models there exists a real edge-coloring model 
such that their partition functions coincide, answering Szegedy's question.

The organization of this paper is as follows. 
In the next section we give definitions of partition functions of edge and vertex-coloring models and state our main 
result (cf. Theorem \ref{thm:main1}). 
In Section \ref{sec:edge-ref} we show, as an application of our main result, that the number of homomorphisms into a simple graph $G$, is not equal to the partition function of a real edge-coloring model unless $G$ is the disjoint union of isolated vertices and complete bipartite graphs with equal sides and we discuss edge-reflection positivity. 
Section \ref{sec:proof} is devoted to proving Theorem \ref{thm:main1}.

\section{Partition functions of edge and vertex-coloring models}\label{sec:prelim}
We give the definitions of edge and vertex-coloring models and their partition functions.
After that we describe Szegedy's result on how to obtain a complex edge-coloring model
from a vertex-coloring model such that their partition functions are the same.
(The existence also follows from the characterization of partition functions of complex edge-coloring
models given in \cite{DGLRS}, but Szegedy gives a direct way to construct the edge-coloring model from the vertex-coloring model.)
And finally we will state our main result saying which partition functions of vertex-coloring
models are partition functions of real edge-coloring models. 

Let $\G$ be the set of all graphs, allowing multiple edges and loops.
Let $\C$ denote the set of complex numbers and let $\R$ denote the set of real numbers.
If $V$ is a vector space we write $V^*$ for its dual space, but by $\C^*$ we mean $\C\setminus\{0\}$.
For a matrix $U$ we denote by $U^*$ its conjugate transpose and by $U^T$ its transpose.

Let $\field$ be a field.
An $\field$-valued \emph{graph invariant} or \emph{graph parameter} is a map $p:\G \to \field$ which takes the same values on isomorphic graphs.
A graph parameter $f:\G\to \C$ is called \emph{multiplicative} if $f(\emptyset)=1$ and if $f(G\cup H)=f(G)f(H)$ 
for all $G,H\in \G$, where $G\cup H$ denotes the disjoint union of $G$ and $H$.

Throughout this paper we set $\N=\{0,1,2\ldots\}$ and for $n\in \N$, $[n]$ denotes the set $\{1,\ldots,n\}$.
We will now introduce partition functions of vertex and edge-coloring models.

Let $a\in (\C^*)^n$ and let $B\in \C^{n\times n}$ be a symmetric matrix.  
We call the pair $(a,B)$ an \emph{$n$-color vertex-coloring model}. 
If moreover, $a$ is positive and $B$ is real, then we call $(a,B)$ a \emph{real $n$-color vertex-coloring model}.
When talking about a vertex-coloring model, we will sometimes omit the number of colors.
The \emph{partition function} of an  $n$-color vertex-coloring model $(a,B)$ is the graph invariant $p_{a,B}:\G\to \C$ defined by
\begin{equation}
p_{a,B}(H):=\sum_{\phi:V(H)\to [n]}\prod_{v\in V(H)}a_{\phi(v)}\cdot \prod_{uv\in E(H)} B_{\phi(u),\phi(v)},
\end{equation}
for $H\in \G$.
Clearly, $p_{a,B}$ is multiplicative.

We can view $p_{a,B}$ in terms of weighted homomorphisms.
Let $G(a,B)$ be the complete graph on $n$ vertices (including loops) with vertex weights given by $a$ 
and edge weights given by $B$. 
Then $p_{a,B}(H)$ can be viewed as counting the number of weighted homomorphisms of $H$ into $G(a,B)$.
In this context $p_{a,B}$ is often denoted by $\hom(\cdot,G(a,B))$. 
We will usually use $\hom(\cdot,G)$ if $G$ is an unweighted graph to emphasize that we count ordinary graph homomorphisms.
The vertex-coloring model can also be seen as a statistical mechanics model where vertices serve as particles, 
edges as interactions between particles, and colors as states or energy levels. 

Let for a field $\field$,
\begin{equation}
R(\field):=\field[x_1,\ldots,x_k] 
\end{equation}
denote the polynomial ring in $k$ variables. We will only consider $\field=\R$ and $\field=\C$.
Note that there is a one-to-one correspondence between linear functions $h:R(\field)\to \field$ and maps 
$h:\N^k\to \field$; $\alpha\in \N^k$ corresponds to the monomial 
$x^{\alpha}:=x_1^{\alpha_1}\cdots x_k^{\alpha_k}\in R(\field)$ and the monomials form a basis for $R(\field)$.
We call any $h\in R(\C)^*$ a \emph{$k$-color edge-coloring model}. 
Any $h\in R(\R)^*$ is called a \emph{real $k$-color edge-coloring model}.
When talking about an edge-coloring model, we will sometimes omit the number of colors.
The \emph{partition function} of a $k$-color edge-coloring model $h$ is the graph invariant $p_h:\G\to \C$ defined by
\begin{equation}
p_h(G)= \sum_{\phi: E(G)\to [k]} \prod_{v\in V(G)} h\Big(\prod_{e\in \delta(v)}x_{\phi(e)}\Big),	\label{eq:part edge}
\end{equation}
for $G\in \G$. Here $\delta(v)$ is the multiset of edges incident with $v$. 
Note that, by convention, a loop is counted twice.
Denote the isolated vertex by $K_1$. Then $p_h(K_1)=0$ according to \eqref{eq:part edge} (as $E(K_1)=\emptyset$).
It is however more natural to define $p_h(K_1)=h(1)$ and extend this multiplicatively for disjoint unions of $K_1$'s.
Then $p_h$ is multiplicative.

The edge-coloring model can be considered as a statistical mechanics model, 
where edges serve as particles, vertices as interactions between particles, and colors as states or energy levels. 
Moreover, partition functions of edge-coloring models generalize the number of proper line graph colorings.

We will now describe a result of Szegedy \cite{SZ7} (see also \cite{SZ10}) showing that partition functions of 
vertex-coloring models are partition functions of edge-coloring models. 

Let $(a,B)$ be an $n$-color vertex-coloring model.
As $B$ is symmetric we can write $B=U^TU$ for some $k\times n$ (complex) matrix $U$, where $k=\rk(B)$, the \emph{rank} of $B$ (cf. \cite[Lemma 5.2.4]{GW}), unless $B$ is equal to the all zero matrix.
Let $u_1,\ldots, u_n\in \C^k$ be the columns of $U$.
Define the edge-coloring model $h$ by $h:=\sum_{i=1}^na_i \ev_{u_i}$, where for $u\in \C^k$, $\ev_u\in R(\C)^*$ is 
the linear map defined by $p\mapsto p(u)$ for $p\in R(\C)$.

\begin{lemma}[Szegedy \cite{SZ7}]	\label{lem:spinvert}
Let $(a,B)$ and $h$ be as above. Then $p_{a,B}(G)=p_h(G)$ for every graph $G$. 
\end{lemma}
Although the proof is not difficult we will give it for completeness.
\begin{proof}
Let $G=(V,E)\in \G$. We may assume that $E\neq \emptyset$.
Then $p_h(G)$ is equal to
\begin{align}
&\sum_{\phi:E\to [k]}\prod_{v\in V}h\big(\prod_{e\in \delta(v)}x_{\phi(e)}\big)
=\sum_{\phi:E\to [k]}\prod_{v\in V}\bigg(\sum_{i=1}^n a_i\prod_{e\in \delta(v)}u_i\big(\phi(e)\big)\bigg)	\label{eq:spinvert}
\\
=&\sum_{\phi:E\to [k]}\sum_{\psi:V\to [n]}\prod_{v\in V}  \Big(a_{\psi(v)}\prod_{e\in \delta(v)}u_{\psi(v)}\big(\phi(e)\big)\Big)
\nonumber 
\\
=&\sum_{\psi:V\to [n]}\prod_{v\in V}a_{\psi(v)}\cdot \sum_{\phi:E\to [k]}\prod_{v\in V} \prod_{e\in \delta(v)}u_{\psi(v)}\big(\phi(e)\big)
\nonumber 
\\
=&\sum_{\psi:V\to [n]}\prod_{v\in V}a_{\psi(v)}\cdot \sum_{\phi:E\to [k]}\prod_{vw\in E} u_{\psi(v)}\big(\phi(vw)\big)u_{\psi(w)}\big(\phi(vw)\big)
\nonumber\\
=&\sum_{\psi:V\to [n]}\prod_{v\in V}a_{\psi(v)}\cdot\prod_{vw\in E}\sum_{i=1}^k u_{\psi(v)}(i)u_{\psi(w)}(i)
=\sum_{\psi:V\to[n]}\prod_{v\in V}a_{\psi(v)}\cdot \prod_{vw\in E}B_{\psi(v),\psi(w)}.	
\nonumber	
\end{align}
By definition, the last line of \eqref{eq:spinvert} is equal to $p_{a,B}(G)$. This completes the proof.
\end{proof}

Note that the proof of Lemma \ref{lem:spinvert} also shows that if $h\in R(\C)^*$ is given by
$h=\sum_{i=1}^na_i \ev_{u_i}$ for certain $a\in (\C^*)^n$ and $u_1,\ldots, u_n\in \C^k$, then $p_h$ can be realized as the 
partition function of an $n$-color vertex-coloring model. Namely take $a=(a_1,\ldots, a_n)$ and $B=U^TU$ where $U$ 
is the matrix with columns the $u_i$. 
\\ \quad \\
\indent
Let $(a,B)$ be an $n$-color vertex-coloring model. We say that $i,j\in [n]$ are \emph{twins of $(a,B)$}
if $i\neq j$ and the $i$th row of $B$ is equal to the $j$th row of $B$. If $(a,B)$ 
has no twins we call the model \emph{twin free}.
Suppose now $i,j\in [n]$ are twins of $(a,B)$. 
If $a_i+a_j\neq 0$, let $B'$ be the matrix obtained from $B$ by removing row and column $i$ and
let $a'$ be the vector obtained from $a$ by setting $a'_j:=a_i+a_j$ and then removing the $i$th entry from it.
In case $a_i+a_j=0$, we remove the $i$th and the $j$th row and column from $B$ to obtain $B'$ and 
we remove the $i$th and the $j$th entry from $a$ to obtain $a'$.
Then $p_{a',B'}=p_{a,B}$. 
So for every vertex-coloring model with twins, we can construct a vertex-coloring model with fewer colors
which is twin free and which has the same partition function.

We need a few more definitions to state our main result.
Let $k\in \N$.
For a $k\times n$ matrix $U$ we denote its columns by $u_1,\ldots,u_n$.
Let, for any $k$, $(\cdot, \cdot)$ denote the standard bilinear form on $\C^k$; i.e, $(u,v)=u^Tv$.
We call the matrix $U$ \emph{nondegenerate} if the span of $u_1,\ldots,u_n$ is nondegenerate with respect
to $(\cdot,\cdot)$. In other words, if $\rk(U^TU)=\rk(U)$.
By $O_k(\C)$ we denote the complex orthogonal group, i.e. 
$O_k(\C):=\{g\in \C^{k\times k}\mid (gv,gv)=(v,v) \text{ for all } v\in \C^k\}$.

We can now state our main result.

\begin{theorem}		\label{thm:main1}
Let $(a,B)$ be a twin-free $n$-color vertex-coloring model.
Let $U$ be a $k\times n$ matrix such that $U^TU=B$, with $k=\rk(B)$. 
Then there exists $g\in O_k(\C)$ such that $gU(gU)^*\in \R^{k\times k}$.
For each such $g$ the following are equivalent:
\begin{enumerate}
 \item[  (i)] $p_{a,B}=p_h$ for some real edge-coloring model $h$,
 \item[ (ii)] the set $\big\{\binom{gu_i}{a_i}\mid i=1,\ldots, n\big \}$ 
 is closed under complex conjugation,
 \item[(iii)] $\sum_{i=1}^n a_i\ev_{gu_i}$ is real.
\end{enumerate}
\end{theorem}


In case $B$ is real, there is an easy way to obtain a $k\times n$ matrix $U$, where $k=\rk(B)$,
such that $UU^*\in \R^{k\times k}$ and $U^TU=B$, using the spectral decomposition of $B$.
So by Theorem \ref{thm:main1}, we get the following characterization
of partition functions of real vertex-coloring models that are partition functions of real edge-coloring models.
We will state it as a corollary.

\begin{cor}\label{cor:main}
Let $(a,B)$ be a twin-free real $n$-color vertex-coloring model. 
Then $p_{a,B}=p_h$ for some real edge-coloring model $h$ 
if and only if for each $i\in[n]$ there exists $j\in [n]$ such that
\begin{enumerate}
 \item[(i) ] $a_i=a_j$,
 \item[(ii)] for each eigenvector $v$ of $B$ with eigenvalue $\lambda:\begin{cases} 
                                                                          \lambda>0 & \Rightarrow v_i=v_j,
                                                                          \\
                                                                          \lambda<0  &\Rightarrow v_i=-v_j.
                                                                        \end{cases}$
 
\end{enumerate}
\end{cor}

\section{Homomorphisms into simple graphs and edge-reflection positivity}\label{sec:edge-ref}
We call a graph \emph{twin free} is its adjacency matrix does not contain two equal rows.
In this section we show, as an application of Theorem \ref{thm:main1}, that for each simple twin free graph $G$ that contains a vertex of degree at least $2$, the parameter $\hom(\cdot,G)$ is not the partition function of a real edge-coloring model. 
We moreover introduce the notion of edge-reflection positivity and Szegedy's characterization of
partition functions of real edge-coloring models in terms of multiplicativity and edge-reflection positivity.
After that we shall discuss some consequences of this characterization to the homomorphism numbers. 
\subsection{Homomorphisms into simple graphs}}
\begin{theorem}\label{thm:graphs not real}
Let $G$ be a simple twin free graph.
Then $\hom(\cdot,G)$ is the partition function of a real edge-coloring model if and only if $G$ is the disjoint union of edges and at most one isolated vertex.
\end{theorem}
\begin{proof}
Denote the edge by $K_2$.
It is easy to see by Corollary \ref{cor:main} that $\hom(\cdot,K_2)$ is the partition function of a real edge-coloring model. This was already shown by Szegedy \cite{SZ7}.
This easily extends to the disjoint union of edges and vertices.

To prove the opposite direction, let $A$ be the adjacency matrix of $G$. 
Just as in the derivation of Corollary \ref{cor:main}, using the spectral decomposition of $A$, we can write
$A=U^TU$, where $U$ has a special structure: 
\begin{equation}\label{eq:structure}
\text{each row of }U \text{ is either imaginary zero or real.}
\end{equation}
This follows from the fact that each row of $U$ is equal to the product of $v$ with a square root of $\lambda$, where $\lambda$ is an eigenvalue of $A$ and $v$ a real eigenvector corresponding to that eigenvalue. 

Let $u_1,\ldots,u_n$ be the columns of $U$ and let $h=\sum_{i=1}^n \ev_{u_i}$.
As $G$ is twin free, the $u_i$ are distinct, so Theorem \ref{thm:main1} implies that $\hom(\cdot,G)$ is the partition function of a real edge-coloring model if and only if $h$ is real valued.
Now suppose that $h$ is real-valued.
Let for $d\in \N$, $B_d$ be the multigraph consisting of two vertices connected by $d$ edges.
For convenience we introduce the following notation: for a map $\phi:[d]\to [n]$ let 
$x^{\phi}:=x_1^{|\phi^{-1}(1)|}\cdots x_n^{|\phi^{-1}(n)|}$.
Then, since $h$ is real valued,
\begin{equation} \label{eq:norm1}
p_h(B_d)=\sum_{\phi:[d]\to [n]} \Big(\sum_{i=1}^n x^{\phi}(u_i)\Big)\Big(\sum_{j=1}^n x^{\phi}(u_j)\Big)
=\sum_{\phi:[d]\to [n]} \Big(\sum_{i=1}^n x^{\phi}(u_i)\Big)\Big(\overline{\sum_{j=1}^n x^{\phi}(u_j)}\Big).
\end{equation}
For even $d$, the righthand-side of \eqref{eq:norm1} can be lower bounded as follows:
\begin{equation}\label{eq:norm2}
p_h(B_d)=\sum_{i,j=1}^n \sum_{\phi:[d]\to[n]} x^{\phi}(u_i)x^{\phi}(\overline{u_j})
=\sum_{i,j=1}^n \la u_i,u_j\ra^d 
\geq\sum_{i=1}^n \|u_i\|^{2d},
\end{equation}
since, by \eqref{eq:structure}, $\la u_i,u_j\ra$ is a real number for each $i,j\in [n]$.
(Here $\la\cdot,\cdot\ra$ denotes the standard Hermitian inner product on $\C^k$.)

Now assume that $G$ has a vertex of degree at least $2$. Then
\begin{equation} \label{eq:contradiction}
\text{there exists a vertex } i\in[n] \text{ such that }\|u_i\|>1.
\end{equation}
To see this, suppose to the contrary that all $\|u_i\|$ are at most $1$.
As $h$ is real valued, we know by Theorem \ref{thm:main1}, that $\{u_1,\ldots,u_n\}$ is closed under complex conjugation.
Fix an edge $ij$ of $G$ and choose $j^*$ such that $u_j=\overline{u_{j^*}}$.
By Cauchy-Schwarz, 
\begin{equation}\label{eq:CS}
1=A_{i,j}=(u_i,u_{j})=\la u_i,u_{j^*}\ra \leq \|u_i\|\|u_{j^*}\|,
\end{equation}
which implies that $\|u_i\|=\|u_{j^*}\|=1$. 
Hence $u_{j^*}=u_i$. 
So for each edge $ij$ of $G$ we have $u_i=\overline{u_j}$.
Let $i$ be a vertex of degree at least $2$, and let $j,k$ be distinct neighbours of $i$.
It follows that $u_j=\overline{u_i}=u_k$, but this contradicts the fact that the $u_i$ are 
distinct (as $G$ is twin free). This shows \eqref{eq:contradiction}.

Now \eqref{eq:norm2} implies that $p_h(B_{2l})$ tends to infinity as $l\to \infty$.
But this contradicts the fact that $p_h(B_{2l})=\hom(B_{2l},G)\leq n^2$ for all $l$. 
So we conclude that if $G$ has a vertex of degree at least $2$, then $\hom(\cdot,G)$ is not equal
to the partition function of any real edge-coloring model.
\end{proof}
In view of removing twins, as described just above Theorem \ref{thm:main1}, the following is a direct corollary to Theorem \ref{thm:graphs not real}.
\begin{cor}\label{cor:graphs not real}
Let $G$ be a simple graph.
Then $\hom(\cdot,G)$ is the partition function of a real edge-coloring model if and only if each connected component of $G$ is either a single vertex or a complete bipartite graph with equal sides.
\end{cor}

\begin{rem}\label{rem:vertex weights}
Let $G=([n],E)$ be a simple graph in which some component is not a single vertex or a complete bipartite graph with two equal sides, so that $\hom(\cdot,G)$ is not the partition function of a real edge-coloring model. Then the proof of Theorem \ref{thm:graphs not real} shows that adding vertex weights $a_i>0$ $(i=1,\ldots,n)$ to $G$ this does not change, as $\hom(B_{2l},G)$ remains bounded.
\end{rem}
\subsection{Edge-reflection positivity}
To describe the concept of edge-reflection positivity we need some definitions.
Let $\bigcirc$ denote the pair $(\emptyset, \{1\})$, which we will call the \emph{circle}. 
Although technically it is not a graph, the circle can be thought of as the graph with one edge and no vertices. 
Let $\G'$ be the set consisting of disjoint unions of elements of $\G$ and finitely many circles.
Note that if $h$ is a $k$-color edge-coloring model, then $p_h(\bigcirc)=k$.
For any $l\in \N$, an \emph{$l$-fragment} is an element of $\G'$, which has $l$ of its vertices labeled $1$ up to $l$, 
each having degree one. 
These labeled vertices are called the \emph{open ends} of the fragment. 
An edge connected to an open end is called a \emph{half edge}. 
Let $\F_l$ be the collection of all $l$-fragments.
We can identify $\F_0$ with $\G'$.
Define a gluing operation $*:\F_l\times \F_{l}\to \G'$ as follows: for $F,H\in \F_l$ connect the neighbors of 
identically labeled open ends with an edge and then delete the open ends; the resulting graph is 
denoted by $F*H$.
Note that by gluing two half edges, of which both their endpoints are open ends, one creates a circle.

For any graph invariant $p,$ let $M_{p,l}$ be the $\F_l\times \F_l$ matrix\footnote{This is an infinite matrix whose rows and columns are indexed by $\F_k$.} defined by 
\begin{equation}
M_{p,l}(F,H)=p(F*H),
\end{equation}
for $F,H\in \F_l$.
This matrix is called the \emph{$l$-th edge connection matrix} of $p$.
A graph invariant $p$ for which $M_{p,l}$ is positive semidefinite for each $l\in \N$ is called \emph{edge-reflection positive}. 
We can now state Szegedy's characterization of partition functions of real edge-coloring models.

\begin{theorem}[Szegedy \cite{SZ7}]	\label{thm:SZ}
Let $p:\G'\to \R$ be a graph invariant. 
Then there exists a real edge-coloring model $h$ such that $p_h=p$ if and only if $p$ is multiplicative and edge-reflection positive.
\end{theorem}

In view of Theorem \ref{thm:SZ},
one could consider Theorem \ref{thm:main1} as a characterization
of those partition functions of vertex-coloring models that are edge-reflection positive.
In particular, by Theorem \ref{thm:graphs not real}, Theorem \ref{thm:SZ} implies that for each 
simple twin free graph $G$ which has a vertex of degree at least $2$, there exists $k,t\in \N$, $k$-fragments $F_1,\ldots, F_t$ and $\lambda\in \R^t$ such that
$\sum_{i,j=1}^t \lambda_i\lambda_j \hom(F_i*F_j,G) <0$.
It would be interesting to explicitly find such inequalities.

It is interesting to relate the consequence of this to homomorphism densities.
The \emph{homomorphism density} $t(H,G)$of a graph $H$ in a graph $G$ is equal to $\frac{1}{|V(G)|^{|V(H)|}}\hom(H,G)$. (This is the probability that a random map from $V(H)$ to $V(G)$ is a homomorphism. Equivalently, giving each vertex of $G$ weight $1/|V(G)|$, $t(H,G)$ is equal to the number of homomorphisms into the weighted graph $G$.)
Let for $p\in [0,1]$, ${\bf G}(n,p)$ be the Erd\H{o}s-R\'enyi random graph (i.e. each edge $ij$, $i,j\in [n]$ is selected independently with probability $p$).
By Remark \ref{rem:vertex weights}, with probability tending to one (as $n$ goes to infinity), $t(\cdot,{\bf G}(n,p))$ is not edge-reflection positive. 
But if we let $n$ go to infinity, $t(H,{\bf G}(n,p))\to p^{|E(H)|}$ (for all simple graphs $H$ with probability one).
This limiting parameter is however edge-reflection positive, as it can be represented by the partition function of the $1$-color edge-coloring model $h_p$ defined by $h_p(x^n):=\sqrt{p}^n$ for $n\in \N$.

We note that this example can be generalised quite a bit with the use of $W$-random graphs introduced by Lov\'asz and Szegedy \cite{LS6} (see also \cite{L12}). 

\section{Proof of Theorem \ref{thm:main1}}\label{sec:proof}
In this section we give a proof of Theorem \ref{thm:main1}; it is based on some fundamental results in geometric invariant theory. We first give a few lemma's after which we can give our proof of the main theorem.

We need some definitions and conventions.
For a square matrix $X$, $\tr(X)$ denotes the \emph{trace} of $X$, the sum of the diagonal elements of $X$.
Recall that $O_l(\C)$ denotes the complex orthogonal group.
The real orthogonal group is the subgroup of $O_l(\C)$ given by all real matrices and is denoted by $O_l(\R)$.

Let $W\in \C^{l\times n}$ be any matrix and
consider the function $f_W:O_l(\C)\to \R$ defined by 
\begin{equation}
g\mapsto \tr(W^*g^*gW)=\tr\big((gW)^*gW\big). 
\end{equation}
This function was introduced by Kempf and Ness \cite{KN} in the context of connected reductive linear algebraic
groups acting on finite dimensional vector spaces.
Note that $f_W$ is left-invariant under $O_l(\R)$ and right-invariant under 
$\stab(W):=\{g\in O_l(\C)\mid gW=W\}$.
Let $e\in O_l(\C)$ denote the identity.
We are interested in the situation that the infimum of $f_W$ over $O_l(\C)$ is equal to $f_W(e)$.

\begin{lemma}	\label{lem:function}
The function $f_W$ has the following properties:
\begin{description}
\item[(i) ]$\inf_{g\in O_l(\C)}f_W(g)=f_W(e)$ if and only if $WW^*\in \R^{l \times l}$,
 \item[(ii)] If $WW^*\in \R^{l\times l}$, then $f_W(e)=f_W(g)$ if and only if $g\in O_l(\R)\cdot \stab(W)$.
\end{description}
 \end{lemma}
\begin{proof}
We start by showing that 
\begin{equation}
f_W  \text{ has a critical point at }e\text{ if and only if }WW^*\in \R^{l\times l}.	\label{eq:critical}
\end{equation}
By definition, a critical point of $f_W$ is a point $g$ such that $(Df_W)_g(X)=0$ for all $X\in T_g(O_l(\C))$, 
where $T_g(O_l(\C))$ is the tangent space of $O_l(\C)$ at $g$ and where $(Df_W)_g$ is the derivative
of $f_W$ at $g$.
It is well known that the tangent space of $O_l(\C)$ at $e$ is the space of skew-symmetric matrices, i.e. 
$T_e(O_l(\C))=\{X\in \C^{l\times l}\mid X^T+X=0\}$.
It is easy to see that the derivative of $f_W$ at $e$ is the $\R$-linear map 
$(Df_W)_e\in \text{Hom}_{\R}(\C^{l\times l},\R)$ defined by
$Z\mapsto \tr(W^*(Z+Z^*)W)$.
Now let $Z$ be skew-symmetric and write $Z=X+\mathfrak{i}Y$, with $X,Y\in \R^{l\times l}$. 
Note that $Z$ is skew-symmetric if and only if both $X$ and $Y$ are skew-symmetric.
Write $W=V+\mathfrak{i}T$ with $V,T\in \R^{l\times l}$. 
Then $(Df_W)_e(Z)$ is equal to 
\begin{align}
&\tr\Big((V^T-\mathfrak{i}T^T)(X+\mathfrak{i}Y+X^T-\mathfrak{i}Y^T)(V+\mathfrak{i}T)\Big)=2\tr\Big((V^T-\mathfrak{i}T^T)\mathfrak{i}Y(V+\mathfrak{i}T)\Big)=		\notag
\\
&2\tr(T^TYV)-2\tr(V^TYT)=4\tr(T^TYV)=4\tr(YVT^T),
\end{align}
where we use that $X$ and $Y$ are skew symmetric, and standard properties of the trace.
So $Df_e(Z)=0$ for all skew symmetric $Y$ if and only if $TV^T=VT^T$.
That is, if and only if $WW^*\in \R^{l\times l}$. 
This shows \eqref{eq:critical}.

By a result of Kempf and Ness (cf. \cite[Theorem 0.1]{KN}) we can now conclude that (i) and (ii) hold.
However, we will give an independent and elementary proof. 

First the proof of $(i)$.
Note that \eqref{eq:critical} immediately implies that $f_W$ does not attain a minimum at $e$ 
if $WW^*\notin \R^{l\times l}$. 
Conversely, suppose $WW^*\in \R^{l \times l}$.
Since $WW^*$ is real and positive semidefinite there exists a real matrix $V$ such that
$WW^*=VV^T$. 
Now note that, by the cyclic property of the trace, $f_W(g)=\tr(g^*gWW^*)$. So we have $f_W=f_V$.
Let $I$ denote the identity matrix.
Take any $g=X+\mathfrak{i}Y\in O_l(\C)$, where $X,Y\in \R^{l\times l}$.
Using that $X^TX-Y^TY=I$, and the fact that $f_W$ is real valued, we find that
\begin{equation}
f_W(g)=\tr\Big((X^TX+Y^TY)VV^T\Big)=\tr(VV^T)+2\tr(Y^TYVV^T)\geq \tr(VV^T)=f_W(e). \label{eq:inner product}
\end{equation}
This proves (i).

Next, suppose that $f_W(g)=f_W(e)$ for some $g\in O_l(\C)$. 
Again, since $WW^*$ is real and positive semidefinite there exists a real matrix $V$ such that
$WW^*=VV^T$. 
Moreover, the span of the columns of $V$ is equal to the span of the columns of $W$.
This implies that $\stab(V)=\stab(W)$.
Now write $g=X+\mathfrak{i}Y$, with $X,Y\in \R^{l\times l}$.
As, by \eqref{eq:inner product}, $f_W(g)=f_W(e)$ if and only if $YV=0$,
it follows that $gV=XV+\mathfrak{i}YV=XV$ is a real matrix.
Let $v_1,\ldots, v_n$ be the columns of $V$.
Then, since by definition of the orhogonal group, $(gv_i,gv_j)=(v_i,v_j)$ for all $i,j$, and since the $gv_i$ are real,
there exists $g_1\in O_l(\R)$ such that $g_1gV=V$.
This implies that $g\in O_l(\R)\cdot \stab(V)$. 
This finishes the proof of (ii).
\end{proof}

For any $l$ and $a\in \C^l$ we denote by $\overline{a}$ the complex conjugate of $a$.

\begin{lemma} 	\label{lem:complex conj}
Let $u_1,\ldots,u_n \in \C^k$ be distinct vectors, let $a\in (\C^*)^n$ and let $h:=\sum_{i=1}^n  a_i \ev_{u_i}$. 
Then $h$ is a real edge-coloring model if and only if the set 
$\big \{\binom{u_i}{a_i}\mid i=1,\ldots, n\big \}$ is closed under complex conjugation.
\end{lemma}

\begin{proof}
Suppose first that the set $\big\{ \binom{u_i}{a_i}\mid i=1,\ldots, n\big \}$ is closed under complex conjugation.
Then for $p\in R(\R)$, $h(p)=\sum_{i=1}^n a_ip(u_i)=\sum_{i=1}^n\overline{a_ip(u_i)}=\overline{h(p)}$. 
Hence, $h(p)\in \R$. So $h$ is real valued.

Now the 'only if' part.
By possibly adding some vectors to $\{u_1,\ldots,u_n\}$ and extending the vector $a$ with zero's,
we may assume that $\{u_1,\ldots,u_n\}$ is closed under complex conjugation.
We must show that $u_i=\overline{u_j}$ implies $a_i=\overline{a_j}$.
We may assume that $u_1=\overline{u_2}$.
Using Lagrange interpolating polynomials we find $p\in R(\C)$ such that $p(u_j)=1$ if $j=1,2$ and $0$ else.
Let $p':=1/2(p+\overline{p})$. Then $p'\in R(\R)$ and consequently, $h(p')=\sum_{i=1}^n a_i p(u_i)=a_1+a_2\in \R$.
Similarly, there exists $q\in R(\C)$ such that $q(u_1)=\mathfrak{i}$, $q(u_2)=-\mathfrak{i}$ and $q(u_j)=0$ if $j>2$.
Setting $q':=1/2(q+\overline{q})$ and applying $h$ to it, we find that $\mathfrak{i}(a_1-a_2)\in \R$.
So we conclude that $a_1=\overline{a_2}$.
Continuing this way proves the lemma.
\end{proof}

We next develop some framework and ideas from \cite{DGLRS} (see also \cite{DG}).
For any $l\in \N$, define
\begin{equation}
S:=\C[y_{\alpha}\mid \alpha\in \N^l],
\end{equation}
the polynomial ring in the infinitely many variables $y_{\alpha}$. These
variables are in bijective correspondence with the monomials of $R(\C)$
via $y_\alpha \leftrightarrow x_1^{\alpha_1} \cdots x_l^{\alpha_l}$.
Let $\N^l_d=\{\alpha\in \N^l\mid |\alpha|\leq d\}$ and let $S_d\subset S$ be the ring of 
polynomials in the (finitely many) variables $y_{\alpha}$ with 
$\alpha\in \N^l_d$. 
Furthermore, let $\G_d$ be the set of all graphs of maximum degree at most $d$.
Let $\C\G$ be the vector space consisting of (finite) formal $\C$-linear combinations of graphs and let
$\pi:\C \G\to S$ be the linear map defined by 
\begin{equation}
G \mapsto \sum_{\phi:EG\to [l]} \prod_{v\in VG} y_{\phi(\delta(v))},
\end{equation}
for any $G\in \G$, where we consider the multiset $\phi(\delta(v))$ as an element of $\N^l$.
Note that $\pi(G)(y)=p_y(G)$ for all $G\in \G$ and $y\in R(\C)^*$.

The orthogonal group acts on $S$ via the bijection between the variables of $S$ and the monomials of $R(\C)$.
Then, as was shown by Szegedy \cite{SZ7} (see also \cite{DGLRS}), for any $d$, 
\begin{equation} 	\label{eq:graphsgeninvs}
\pi(\C \G_d)=S_d^{O_l(\C)},
\end{equation}
where $S_d^{O_l(\C)}$ denotes the subspace of $S^d$ of polynomials that are $O_l(\C)$-invariant.
Note that the action of $O_l(\C)$ on $R(\C)$ induces an action on $R(\C)^*$, i.e. $O_l(\C)$ acts on edge-coloring models.
Then \eqref{eq:graphsgeninvs} in particular implies that $p_{gy}=p_y$ for all $g\in O_l(\C)$ and all $y\in R(\C)^*$.

Let, for any $d$,
\begin{equation}
Y_d:=\{y\in \C^{\N^l_d}\mid \pi(G)(y)=p_{h}(G) \text{ for all } G\in \G_d\}.
\end{equation}
Then $Y_d$ is a fiber of
the quotient map $\C^{\N^l_d} \to \C^{\N^l_d}//O_l(\C)$. In particular, $Y_d$ 
contains a unique closed orbit $C_d$ (cf. \cite[Section 8.3]{Hum} or \cite[Satz 3, page 101]{kraft}).

Let $\pr_d:\C^{\N^l}\to \C^{\N^l_d}$ be the
projection sending $y$ to $y_d:=y|_{\C^{\N^l_d}}$. 
We also write $\pr_d$ for the restriction of $\pr_d$ to $\C^{\N_{d'}}$, for any $d'\geq d$. 
Note that $\pr_d(Y_{d'})\subseteq Y_d$ for $d'\geq d$, as $\G_d\subseteq \G_{d'}$.  

We can consider any $k$-color edge-coloring model $y$ as an $l$-color edge-coloring model without 
changing its partition function on $\G$, by setting $y(\alpha)=0$ if $\alpha_i>0$ for some $i>k$.
The following lemma is based on results from \cite{DG}.

\begin{lemma}  \label{lem:orbit}
Let $h:=\sum_{i=1}^n a_i\ev_{u_i}\in R(\C)^*$, with $a\in (\C^*)^n$ and distinct $u_1,\ldots,u_n\in \C^k$.
Suppose the bilinear form restricted to the span of the $u_i$ is nondegenerate.
If $y$ is a real $l$-color edge-coloring model such that $p_{h}(G)=p_y(G)$ for all $G\in \G$,
then there exists $g\in O_{l}(\C)$ such that $gh=y$.
\end{lemma}

\begin{proof}
We may assume that $l\geq k$.
In case $l>k$, we need to append the $gu_i$'s with $l-k$ zero's.
Note that the bilinear form restricted to the span of the $u_i$ remains nondegenerate.

Then, by \cite[Theorem 5]{DG}, for each $d\geq 3n$, $h_d\in C_d$.
Now since $y$ is real valued, a result of Kempf and Ness \cite[Theorem 0.2]{KN} 
(see also \cite[Proposition 7.9]{R13}) implies that $y_d\in C_d$, for every $d$.
We now claim that this implies that there exists $g\in O_l(\C)$ such that $gh=y$.
Indeed, define, for any $d$, the stabilizer of $y_d$ by 
\begin{equation}
\stab(y_d):=\{g\in O_l(\C)\mid gy_d=y_d\}.
\end{equation}
Then $\stab(y_d)=\cap_{d'\leq d}\stab(y_{d'})$.
Since $O_l(\C)$ is Noetherian there exists $d_1\geq 3n$ such that 
$\stab(y_{d_1})=\cap_{d\in \N}\stab(y_d)$.
Now since we have a canonical bijection from $O_l(\C)/\stab(y_d)$ to $C_d$, 
this implies that for any $d\geq d_1$, if $g\in O_l(\C)$ is such that $gy_d=h_{d}$, then also $gy=h$.
This proves the lemma.
\end{proof}

\begin{proof}[Proof of Theorem \ref{thm:main1}]
Recall that $U$ is a $\rk(B)\times n$ matrix such that $U^TU=B$, with columns $u_1,\ldots,u_n$. 
It is well known that since the matrix $U$ is nondegenerate, 
the $O_k(\C)$-orbit of $U$ is closed (cf. \cite[Theorem 5]{DG}).
This implies that $f_U$ attains its minimum at some $g\in O_k(\C)$.
So Lemma \ref{lem:function} (i) implies that $gU(gU)^*\in \R^{k\times k}$.

Let $h'=\sum_{i=1}^n a_i \ev_{gu_i}$ and let $h=\sum_{i=1}^n a_i \ev_{u_i}$.
Observe that since $(gU)^TgU=B$, Lemma \ref{lem:spinvert} implies that $p_h=p_{h'}$. 
This shows that (iii) implies (i). (This also follows from \eqref{eq:graphsgeninvs}, using that $h'=gh$.)
Moreover, for the rest of the proof we may assume that $g$ is equal to the identity.

Since $(a,B)$ is twin free, the $u_i$ are distinct. Hence Lemma \ref{lem:complex conj} immediately implies the equivalence of (ii) and (iii). 

To prove that (i) implies (iii).
Let $y$ be a real $l$-color edge-coloring model such that $p_{a,B}=p_y$.
Since $U$ is nondegenerate, Lemma \ref{lem:orbit} implies the existence of a $g\in O_{l}(\C)$ such that $y=gh$.
Now note that $y=\sum_{i=1}^n a_i \ev_{gu_i}$. 
As $y$ is real, Lemma \ref{lem:complex conj} implies that the set $\{gu_i\}$ is closed under complex conjugation, implying that $gU(gU)^*\in \R^{l \times l}$.
So by Lemma \ref{lem:function} (i) the infimum of $f_{gU}$ is attained at $e$.
Equivalently, the infimum of $f_U$ is attained at $g$.
Since $UU^*\in \R^{k\times k}$, this implies, by Lemma \ref{lem:function} (ii), that $g\in O_l(\R)\cdot\stab(U)$.
Hence $g=g_1\cdot s$ for some $g_1\in O_l(\R)$ and $s\in \stab(U)$.
Now note that since $sh=h$ we have that $h=g_1^{-1}y$ and hence $h$ is real.
\end{proof}

\begin{ack}
I thank Lex Schrijver for his comments on earlier versions of this paper. 
In particular, for simplifying some of the proofs. I moreover thank the anonymous referees for their comments, improving the 
readability of this paper.
\end{ack}

\end{document}